\documentclass[11pt,reqno]{amsart}

\usepackage{enumerate,url,mathrsfs}
\usepackage{amsfonts}
\usepackage{graphicx}
\usepackage{esint,comment}
\pdfoutput=1

\newtheorem{theorem}{Theorem}[section]

\newtheorem*{lemma*}{Lemma}
\newtheorem{proposition}[theorem]{Proposition}
\newtheorem{corollary}[theorem]{Corollary}

\theoremstyle{definition}

\theoremstyle{remark}

\numberwithin{equation}{section}

\def\XXint#1#2#3{{\setbox0=\hbox{$#1{#2#3}{\int}$}
\vcenter{\hbox{$#2#3$}}\kern-.5\wd0}}

\begin{document}
\baselineskip4mm
\vskip0.1cm
\title[]{A class of positive Fox $H$-functions
}

\author[Filippo Giraldi]{Filippo Giraldi}
\address{Section of Mathematics, International Telematic University Uninettuno, \\ Corso Vittorio Emanuele II 39, 00186, Rome, Italy\\
School of Chemistry and Physics, University of KwaZulu-Natal \\ Westville Campus, Durban, 
South Africa}





\begin{abstract}
The Fox $H$-function is a special function which is defined via the Mellin-Barnes integrals and produces, as particular cases, Wright generalized hypergeometric functions, MacRobert's $E$-functions and Meijer $G$-functions, to name but few. Various cases of non-negative Fox $H$-functions are obtained in literature by relying on the properties of integral transforms and the complete monotonicity. In the present scenario, Fox $H$-functions, which are positive on $\mathbb{R}^+$, are determined via the Mellin convolution products of finite combinations, with possible repetitions, of elementary functions. The chosen elementary functions are non-negative on $\mathbb{R}^+$ and are defined via stretched exponential and power laws. Further forms of positive Fox $H$-functions can be obtained from the former via elementary properties and integral transforms. As particular cases, we determine forms of Wright generalized hypergeometric functions, MacRobert's $E$-functions and Meijer $G$-functions which are positive on $\mathbb{R}^+$.
\end{abstract}

\maketitle


\section{Introduction}\label{1}


The Fox $H$-function is a special function which is defined via the Mellin-Barnes integrals and provides various special functions as particular cases \cite{Fox1961,SriGuGo1982,PrBrMa1986,KilSa2010,MathSaxHau2010}. In fact, Wright generalized hypergeometric functions, MacRobert's $E$-functions and Meijer $G$-functions, are particular cases of the Fox $H$-function, to name but few. The Fox $H$-function is adopted for the description of the most various phenomena. In this regard, see Refs. \cite{MathSax1978,MathSaxHau2010,MaPaSa2005}, to name but few. For example, the Fox $H$-function is widely used in statistics. As matter of facts, the expectation values of statistically independent random variables are given by positive Fox $H$-functions in case the statistics is described by generalized gamma densities or type-1 and type-2 beta densities \cite{MathSaxHau2010,MathSax1978}. 

The property of non-negativity is essential for the use of the Fox $H$-function as distribution or density. In this regard, see Refs. \cite{CartSprin1977,BegCriDaSil2024}, to name but few. In literature, non-negative Fox $H$-functions and Wright generalized hypergeometric functions, are determined in various ways, for example, by relying on complete monotonicity and the properties of integral transforms. See Refs. \cite{Meh2021,ElKaMeh2010,Meh2018,LucKir2013}, to name but few.  

As a continuation of the above-described scenario, here, we aim to determine further conditions on the involved indexes and parameters such that the Fox $H$-function is positive on $\mathbb{R}^+$. The paper is organized as follows. A brief introduction of the Fox $H$-function, or the Mellin transform, is reported in Section \ref{2}, or \ref{3}. In Section \ref{4}, we determine Fox $H$-functions which are positive on $\mathbb{R}^+$ via Mellin convolution products of elementary functions. In Section \ref{5}, further forms of positive Fox $H$-functions are obtained from elementary properties or integral transforms of positive Fox $H$-functions. In Section \ref{6}, we consider special functions, which are particular cases of the Fox $H$-function, and we determine conditions which make these functions positive on $\mathbb{R}^+$. Summary of the results and conclusions are reported in Section \ref{7}. 

\section{The Fox $H$-function}\label{2}

The Fox $H$-function is defined via Mellin-Barnes integrals \cite{Fox1961,SriGuGo1982,PrBrMa1986,KilSa2010,MathSaxHau2010}:
  \begin{eqnarray}
&&\hspace{-2.5em}H_{p,q}^{m,n}\left[z\Bigg|
\begin{array}{rr}
\left(\alpha_j,A_j\right)_1^p\\
\left(\beta_j,B_j\right)_1^q
\end{array}
\right]=\frac{1}{2 \pi \imath} 
\int_{\mathcal{C}}
\Xi_{p,q}^{m,n}\left[s\Bigg|
\begin{array}{rr}
\left(\alpha_j,A_j\right)_1^p\\
\left(\beta_j,B_j\right)_1^q
\end{array}
\right]
 \,z^{-s}  ds, \label{FoxHdef}
\end{eqnarray}
where
\begin{eqnarray}
&&\hspace{-2.5em}
\Xi_{p,q}^{m,n}\left[s\Bigg|
\begin{array}{rr}
\left(\alpha_j,A_j\right)_1^p\\
\left(\beta_j,B_j\right)_1^q
\end{array}
\right]
=
\frac{\prod_{j=1}^m \Gamma\left(\beta_j+B_j s\right)
\prod_{j=1}^n \Gamma\left(1-\alpha_j-A_j  s\right)
}{\prod_{j=m+1}^q \Gamma\left(1-\beta_j-B_j  s\right)
\prod_{j=n+1}^p \Gamma\left(\alpha_j+A_j  s\right)
},\nonumber \\ && \label{Thetas}
\end{eqnarray}
in case the poles of the Gamma functions $\Gamma\left(\beta_1+B_1 s\right),\ldots$, $\Gamma\left(\beta_m+B_m s\right)$, differ from the poles of the Gamma functions $\Gamma\left(1-\alpha_1-A_1 s\right),\ldots$, \\$\Gamma\left(1-\alpha_n-A_n s\right)$. The required conditions are provided by the following constraints: 
\begin{eqnarray}
A_j\left(l+\beta_{j'}\right)\neq B_{j'}\left(\alpha_{j}-l'-1\right), \label{cond1}
\end{eqnarray}
for every $j=1, \ldots, n$, $j'=1, \ldots, m$, and $l,l' \in \mathbb{N}_0$, where $\mathbb{N}_0$ is the set of natural numbers, $\mathbb{N}_0\equiv \mathbb{N}\cup \left\{0\right\}$ and $\mathbb{N}\equiv \left\{1,2,
\ldots\right\}$. The empty products are interpreted as unity. The indexes $m,n,p,q$ take the following values: $0\leq n\leq p$, $0\leq m\leq q$, and the involved parameters are characterized by the following properties: $A_i, B_j\in \mathbb{R}^+$, $\alpha_i,\beta_j \in \mathbb{C}$, for every $i=1,\cdots,p$, and $j=1,\cdots,q$, where $\mathbb{R}^+$ is the set of the positive real numbers and $\mathbb{C}$ is the set of complex numbers. The following notation is adopted for the sake of shortness: the form $\left(x_j,X_j\right)_1^l$ is equivalent to the following expression: $\left(x_1,X_1\right),\ldots,\left(x_l,X_l\right)$, for every $l\in \mathbb{N}$. The term $\left(x_j,X_j\right)_1^l$ is unity in case the natural number $l$ vanishes. Similarly, the form $\left(\left(x_j,X_j\right),\left(y_j,Y_j\right)\right)_1^l$ is equivalent to the following expression: $\left(x_1,X_1\right)$, $\left(y_1,Y_1\right)$, $\ldots$, $\left(x_l,X_l\right),\left(y_l,Y_l\right)$, for every $l\in \mathbb{N}$. Also, the form $\left(x_j\right)_1^l$ is equivalent to the following expression: $x_1$, $\ldots$, $x_l$, for every $l\in \mathbb{N}$.

The main properties of the Fox $H$-function are determined by the parameters $\chi$, $\mu$, $\delta$, $\kappa$, defined as below \cite{KilSa2010,MathSaxHau2010},
\begin{eqnarray}
&&\chi=\sum_{j=1}^n A_j-\sum_{j=n+1}^p A_j
+\sum_{j=1}^m B_j-\sum_{j=m+1}^q B_j, \label{chi} \\
&& \mu=\sum_{j=1}^q B_j-\sum_{j=1}^p A_j,  \label{mu}\\
&&\delta=\sum_{j=1}^q \beta_j-\sum_{j=1}^p \alpha_j + \frac{p-q}{2}, \label{delta}
\\&&\kappa=\left[\prod_{j=1}^n \left(A_j\right)^{-A_j}\right]
\left[\prod_{j=1}^n \left(B_j\right)^{B_j}\right]. \label{kappa} 
 \end{eqnarray}
For the sake of shortness, refer to \cite{KilSa2010} for details on the existence condition, the contour path $\mathcal{C}$, the domain of analiticity and the properties of the Fox $H$-function. 

\section{The Mellin transform}\label{3}
 
Consider a general complex-valued function $g$, defined on the set of positive real numbers, $g: \mathbb{R}^+\to \mathbb{C} $. The Mellin transform $\mathfrak{M}\left[g\right](s)$, or $\hat{g}(s)$, of the function $g(s)$ is defined as below,
\begin{eqnarray}\mathfrak{M}\left[g\right](s)=\hat{g}(s)\equiv\int_0^{+\infty}
g(t) t^{s-1} dt
, \label{Mg}
\end{eqnarray}
for every complex value of the variable $s$ such that the above-reported integral exists. Let $g_1,g_2: \mathbb{R}^+\to \mathbb{C}$ be functions such that 
$t^{s-1} g_j(t)\in L_1\left(\mathbb{R}^+\right)$ for every complex value of the variable $s$ such that $x_j \leq \Re \left(s\right) \leq x'_j$, for every $j=1,2$, where $L_1\left(\mathbb{R}^+\right)$ is the space of the functions which are summable on $\mathbb{R}^+$. The Mellin transforms $\hat{g}_1(s)$ and $\hat{g}_2(s)$ coincide on a nonempty strip,
\begin{eqnarray}
\hat{g}_1(s)=\hat{g}_2(s),
 \label{mg1eqmg2}
\end{eqnarray}
for every $s$ such that 
$x \leq \Re \left(s\right) \leq x'$, where $x= \max 
\left\{x_1,x_2 \right\}$, $x'= \min 
\left\{x'_1,x'_2 \right\}$, and $x<x'$. If the above-reported conditions hold, the functions $g_1(t)$ and $g_2(t)$ are equal almost everywhere on $\mathbb{R}^+$ \cite{GlaePrudSkorn2006},
\begin{eqnarray}
g_1(t) \doteq g_2(t),
 \label{g1eqg2ae}
\end{eqnarray}
for $t>0$. The symbol $\doteq$ means that the equality holds almost everywhere on $\mathbb{R}^+$. Let the functions $g_1(t)$ and $g_2(t)$ be continuous over the non-empty interval $\left[t',t''\right]$, with $0<t'< t''$. If the above-reported conditions hold, the functions $g_1(t)$ and $g_2(t)$ are equal over the interval $\left[t',t''\right]$,
 \begin{eqnarray}
g_1(t) = g_2(t),
 \label{g1eqg2}
\end{eqnarray}
for every $t$ such that $t'\leq t \leq t''$.

The Mellin convolution product $\left(g_1 \vee g_2\right)(t)$ of the functions $g_1(t)$ and $g_2(t)$ is defined as below \cite{TitchmarshFT,Widder,DoetschHLT,MarichevMT,ButJan1997,GlaePrudSkorn2006}, 
\begin{eqnarray}
\left(g_1 \vee g_2\right)(t)\equiv \int_0^{+\infty}\frac{1}{\tau}\,g_1\left(\frac{t}{\tau}\right)g_2\left(\tau\right) d\tau
, \label{g1g2uMellinConv}
\end{eqnarray}
for every $t>0$, in case the above-reported integral exists. The Mellin convolution product is commutative and 
associative. Let the following condition: 
$t^{x_0-1} g_j(t)\in L_1\left(\mathbb{R}^+\right)$ hold for the real value $x_0$ such that 
$x \leq x_0 \leq x'$ and every $j=1,2$. If the above-reported conditions hold, the Mellin convolution product is Mellin transformable for $\Re \left(s\right) =x_0$, $t^{x_0-1} \left(g_1 \vee g_2\right)(t)\in L_1\left(\mathbb{R}^+\right)$, and 
\begin{eqnarray}
\mathfrak{M}\left[\left(g_1 \vee g_2\right)\right](s)=\hat{g}_1(s)
\hat{g}_2(s),
 \label{MTg1convg2}
\end{eqnarray}
for $\Re \left(s\right)=x_0$ \cite{TitchmarshFT,Widder,DoetschHLT,MarichevMT,ButJan1997,GlaePrudSkorn2006}. Additionally, let the function $t^{x_0}g_1(t)$ be uniformly continuous and bounded on $\mathbb{R}^+$. Then, the convolution product $\left(g_1 \vee g_2\right)(t)$ is continuous on $\mathbb{R}^+$ \cite{ButJan1997}.

The Mellin transform of the Fox $H$-function is given by the expression below \cite{KilSa2010}, 
\begin{eqnarray}
&&\hspace{-
6.0em}
\int_{0}^{+\infty}
t^{s-1} H_{p,q}^{m,n}\left[ t\Bigg|
\begin{array}{rr}
\left(\alpha_j,A_j\right)_1^{p}\\
\left(\beta_j,B_j\right)_1^{q}
\end{array}
\right]dt=
\Xi_{p,q}^{m,n}\left[s\Bigg|
\begin{array}{rr}
\left(\alpha_j,A_j\right)_1^{p}\\
\left(\beta_j,B_j\right)_1^{q}
\end{array}
\right],
\label{MTFoxH}
\end{eqnarray}
for the following values of the complex variable $s$:
\begin{eqnarray}
-\min_{j=1,\ldots, m}\left\{\frac{\Re \left(\beta_j\right)}{B_j}\right\}
<\Re \left(s\right)<
\min_{j=1,\ldots, n}\left\{\frac{1-\Re \left( \alpha_j \right)}{A_j}\right\}
, \label{MTHcond1}
 \end{eqnarray}
if $\chi>0$. 


\section{Positive Fox $H$-functions via Mellin convolution products }\label{4}

Consider the real-valued function $f: \mathbb{R}^+\to \mathbb{R}^+_0$ which is defined by the following Mellin convolution products:
\begin{eqnarray}
 &&\hspace{-5em}f(t)\equiv \Big(\Upsilon_{n_1}\vee \Phi_{n_2} \vee
\Psi_{n_3} \vee \Lambda_{n_4}\Big)(t)
, \label{ftMellinConvTrunc}
\end{eqnarray}
where $\left(n_1,n_2,n_3,n_4 \right)\in \mathbb{N}_0^4$, and $\mathbb{R}^+_0 \equiv \mathbb{R}^+ \cup \left\{0\right\}$. The functions $\Upsilon_{n_1}$, $\Phi_{n_2}$, $\Psi_{n_3}$, $\Lambda_{n_4} : \mathbb{R}^+\to \mathbb{R}^+_0$ are defined via the following Mellin convolution products:
\begin{eqnarray}
 &&\hspace{-5em}\Upsilon_{n}(t)\equiv\left(\varphi_{a_1,b_1}\vee \cdots \vee \varphi_{a_{n},b_{n}}\right)(t), \label{Upsilon}\\
&&\hspace{-5em}\Phi_{n}(t)\equiv \left(\phi_{a'_1,c_1,d_1}\vee \cdots \vee \phi_{a'_{n},c_{n},d_{n}}\right)(t)
, \label{Phi}\\
&&\hspace{-5em}\Psi_{n}(t)\equiv\left( \psi_{a''_1,o_1,r_1}\vee \cdots \vee \psi_{a''_{n},o_{n},r_{n}}\right)(t)
, \label{Psi}\\
&&\hspace{-5em}\Lambda_{n}(t)\equiv\left(\eta_{a'''_1,v_1,w_1}\vee \cdots \vee \eta_{a'''_{n},v_{n},w_{n}}
\right)(t)
, \label{Lambda}
\end{eqnarray}
for every $n\in \mathbb{N}$. By definition, the functions $\Upsilon_{0}(t)$, $\Phi_{0}(t)$, $\Psi_{0}(t)$, $\Lambda_{0}(t)$ coincide with unity, 
\begin{eqnarray}
 &&\hspace{-5em}\Upsilon_{0}(t)\equiv 1, \,\,\Phi_{0}(t)\equiv 1, \,\,\Psi_{0}(t)\equiv 1,\,\,\Lambda_{0}(t)\equiv 1, \label{convn0}
\end{eqnarray}
for every $t\in \mathbb{R}^+$. The functions $\varphi_{a,b}(t),\phi_{a,b,c}(t),\psi_{a,b,c}(t),\eta_{a,b,c}(t)$ are elementary functions which are defined via stretched exponential and power laws. Particularly, the function $\varphi_{a,b}: \mathbb{R}^+\to \mathbb{R}^+$ is defined by the following form:
\begin{eqnarray}
\varphi_{a,b}(t)\equiv\frac{t^{b/a}}{a }\, \exp\left(-t^{1/a}\right), \hspace{1em}t> 0,\label{varphiudef}
\end{eqnarray}
for every $a>0$, $b\geq 0$. The function $\phi_{a,b,c}: \mathbb{R}^+\to \mathbb{R}^+_0$ is defined by the expression below,
\begin{eqnarray}
&&\hspace{-3em}\phi_{a,b,c}(t)\equiv
\frac{t^{b/a}\left(1-t^{1/a}\right)^{c-b-1}}{a \Gamma\left(c-b\right)},\hspace{1em}0<t<1,\nonumber \\&&\label{phiudef}\\
&&\hspace{-3em}\phi_{a,b,c}(t)\equiv 0,\hspace{1em}t\geq 1,\nonumber
\end{eqnarray}
for every $a>0$, $b\geq 0$, $c\geq b+1$. The function $\psi_{a,b,c}: \mathbb{R}^+\to \mathbb{R}^+$ is defined by the following form:
\begin{eqnarray}
\psi_{a,b,c}(t)\equiv\frac{\Gamma\left(b+c\right)}{a }\, t^{b/a}\left(1+t^{1/a}\right)^{-b-c},\hspace{1em}t>0, \label{psiudef}
\end{eqnarray}
for every $a,c>0$, and $b \geq 0$. The function $\eta_{a,b,c}: \mathbb{R}^+\to \mathbb{R}^+_0$ is defined by the expression below,
\begin{eqnarray}
&&\hspace{-3em}\eta_{a,b,c}(t)\equiv\frac{t^{(1-c)/a}}{a \Gamma\left(c-b\right)}\left(t^{1/a}-1\right)^{c-b-1},\hspace{1em}t>1,\nonumber \\&&\hspace{-3em}\label{etaudef}\\&&\hspace{-3em}
\eta_{a,b,c}(t)\equiv 0, 
\hspace{1em}0<t\leq 1,\nonumber 
\end{eqnarray}
for every $a,b>0$, and $c\geq b+1$. The parameters $a_1,\ldots$, $a_{n}$, $b_1,\ldots$, $b_{n}$, $a'_1,\ldots$, $a'_{n}$, $c_1$,$\ldots$, $c_{n}$, $d_1, \ldots$, $d_{n}$, $a''_1, \ldots$, $a''_{n}$, $o_1$, $\ldots$, $o_{n}$, $r_1,\ldots$, $r_{n}$, $a'''_1,\ldots$, $a'''_{n}$, $v_1,\ldots$, $v_{n}$, $w_{1},\ldots$, $w_{n}$, involved in the definition of the function $f(t)$, are required to fulfill the following constraints:
\begin{eqnarray}
 &&a_j>0,\,\, b_j \geq 0,\hspace{1em}j=1, \ldots, n_1, \label{cond_ab} \\
 &&a'_j>0,\,\, c_j \geq 0,\,\, d_j \geq c_j +1,    \hspace{1em}j=1, \ldots, n_2, \label{cond_apdc}\\&&a''_j, r_j>0,\,\, o_j \geq 0,\hspace{1em}j=1, \ldots, n_3, \label{cond_appro}\\
&&a'''_j,v_j>0,\,\,w_j\geq v_j+1,\hspace{1em}j=1, \ldots, n_4. \label{cond_apppwv}
\end{eqnarray}
for every $\left(n_1,n_2,n_3,n_4 \right)\in \mathbb{S}$, where $\mathbb{S}\equiv \mathbb{N}_0^4\setminus \left\{ \left(0,0,0,0 \right)\right\}$. Constraints (\ref{cond_ab}) - (\ref{cond_apppwv}) are obtained from 
relations (\ref{varphiudef}) - (\ref{etaudef}).

\begin{theorem}
The function $f: \mathbb{R}^+\to \mathbb{R}^+_0$, defined by Eq. (\ref{ftMellinConvTrunc}) and constraints (\ref{cond_ab}) - (\ref{cond_apppwv}), is positive, $f(t)> 0$ for every $t>0$, for every $\left(n_1,n_2,n_3,n_4 \right)\in \mathbb{N}_0^4$, except for the following cases: $n_1,n_3,n_4=0$, and $n_2\geq 1$; or $n_1,n_2,n_3=0$, and $n_4\geq 1$. Particularly, $f(t)>0$, for $0<t<1$, and $f(t)=0$, for $t \geq 1$, if $n_1,n_3,n_4=0$ and $n_2\geq 1$; while $f(t)=0$, for $0<t\leq 1$, and $f(t)>0$, for $t > 1$, if $n_1,n_2,n_3=0$ and $n_4\geq 1$. 
\label{posf}
\end{theorem}

\begin{proof}
Consider the functions $\rho_1: \mathbb{R}^+\to \mathbb{R}^+$, and $\rho_2,\rho_3: \mathbb{R}^+\to \mathbb{R}_0^+$, with the following properties: $\rho_1(t)>0$, for $t>0$; $\rho_2(t)>0$, for $0< t<1$, and $\rho_2(t)=0$, for $t \geq 1$; $\rho_3(t)=0$, for $0< t\leq 1$, and $\rho_3(t)>0$, for $t > 1$. Let the convolution product $\left(\rho_j \vee \rho_k \right)(t)$ exist for every $j,k=1,2,3$, and $t>0$. Then, the following properties hold:
\begin{eqnarray}
&&\left(\rho_2 \vee \rho_2 \right)(t)>0, \hspace{1em}0< t<1, \label{r221} \\
&&\left(\rho_2 \vee \rho_2 \right)(t)=0, \hspace{1em}t\geq 1,  \label{r222} \\
&&\left(\rho_3 \vee \rho_3 \right)(t)>0, \hspace{1em}t\geq 1, \label{r331}  \\
&&\left(\rho_3 \vee \rho_3 \right)(t)=0, \hspace{1em} 0< t<1,\label{r332} \\
&&\left(\rho_j \vee \rho_k \right)(t)>0, 
\hspace{1em}t>0,  \label{rjk} 
\end{eqnarray}
for $j=1$ and $k=1,2,3$; $j=2$ and $k=1,3$; $j=3$ and $k=1,2$. The functions defined by Eqs. (\ref{varphiudef}) - (\ref{etaudef}), and constraints (\ref{cond_ab}) - (\ref{cond_apppwv}), are special cases of the functions $\rho_1,\rho_2,\rho_3$. Thus, the function $f(t)$, given by Eq. (\ref{ftMellinConvTrunc}), is the result of convolution products of functions like $\rho_1$, $\rho_2$, $\rho_3$. Conditions (\ref{r221}) - (\ref{rjk}) and the commutative and associative properties of the Mellin convolution product suggest that the convolution product given by Eq. (\ref{ftMellinConvTrunc}) is positive for every $t>0$, except for 
the two following cases: $n_1,n_3,n_4=0$, and $n_2\geq 1$; or $n_1,n_2,n_3=0$, and $n_4\geq 1$. The convolution product given by Eq. (\ref{ftMellinConvTrunc}) is non-negative in the two latter cases.
\end{proof}

\begin{proposition}
\label{MTft}
The Mellin transform $\hat{f}(s)$ of the function $f(t)$, defined by Eq. (\ref{ftMellinConvTrunc}) and constraints (\ref{cond_ab}) - (\ref{cond_apppwv}), is given by the form below,
\begin{eqnarray}
&&\hspace{-3.3em}\hat{f}(s)= \left[\prod_{j=1}^{n_1}\Gamma\left(a_j s+b_j\right)\right]\left[
\prod_{j=1}^{n_2}\frac{\Gamma\left(a'_j s+c_j\right)}{\Gamma\left(a'_j s+d_j\right)}\right]\nonumber \\ &&\times\left[
\prod_{j=1}^{n_3}\Gamma\left(a''_j s+o_j\right)
\Gamma\left(r_j-a''_j s\right)\right] \left[
\prod_{j=1}^{n_4}\frac{\Gamma\left(v_j-a'''_j s\right)}{\Gamma\left(w_j-a'''_j s\right)}\right],  \label{MTfGammaProducts}
\end{eqnarray}
for every value of the complex variable $s$ such that 
\begin{eqnarray}
\xi'>\Re \left(s\right)> -\xi,
\label{stripMft}
\end{eqnarray}
where
\begin{eqnarray}
&&\xi'= \min\left\{\frac{r_1}{a''_1}, \ldots,\frac{r_{n_3}}{a''_{n_3}},\frac{v_1}{a'''_1}, \ldots,\frac{v_{n_4}}{a'''_{n_4}}\right\}, \label{xip}\\
&&\xi=\min\left\{\frac{b_1}{a_1}, \ldots,\frac{b_{n_1}}{a_{n_1}},
\frac{c_1}{a'_1}, \ldots,\frac{c_{n_2}}{a'_{n_2}},\frac{o_1}{a''_1}, \ldots,\frac{o_{n_3}}{a''_{n_3}}\right\}. \label{xi}  
 \end{eqnarray}
\end{proposition}

\begin{proof}
The Mellin transform $\hat{f}(s)$, is evaluated from the Mellin convolution form (\ref{ftMellinConvTrunc}) as a product of the following Mellin transforms \cite{ButJan1997,MarichevMT}:
\begin{eqnarray}
&&\hspace{-3em}\hat{\varphi}_{a,b}(s)=\Gamma\left(a s +b\right), 
\hspace{1em} \Re \left(s\right)>-\frac{b}{a}, \,\, a>0, \,\,b\geq 0, \label{MTvarphi}\\
&&\hspace{-3em}\hat{\phi}_{a,b,c}(s)=\frac{\Gamma\left(a s +b\right)}{\Gamma\left(a s +c\right)},  
\hspace{1em} \Re \left(s\right)>-\frac{b}{a}, \,\, a>0, \,\,b\geq 0,
\,\,c\geq b+1,\label{MTphi}\\
&&\hspace{-3em}\hat{\psi}_{a,b,c}(s)=
\Gamma\left(a s +b\right)\Gamma\left(c-a s\right),
\hspace{0.3em} \frac{c}{a}>\Re \left(s\right)>-\frac{b}{a},\,a,c>0,\,b \geq 0, \label{MTpsi}\\
&&\hspace{-3em}\hat{\eta}_{a,b,c}(s)
=\frac{\Gamma\left(b-a s\right)}{\Gamma\left(c-a s\right)},  
\hspace{1em} \Re \left(s\right)<\frac{b}{a}, \,\,a,b>0,\,\,c\geq b+1. \label{MTeta}
\end{eqnarray}
The above-reported constraints concerning the value of $\Re \left(s\right)$ provide Eqs. (\ref{stripMft}) - (\ref{xi}). 
\end{proof}

Consider the following Fox $H$-function:
\begin{eqnarray}
 &&\hspace{-2em}H^{m',n'}_{p',q'}\left[t
\Bigg|
\begin{array}{ll}
\left(\alpha'_j,A'_j\right)_1^{p'} \\  
\left(\beta'_j,B'_j\right)_1^{q'}
\end{array}
\right], 
\label{FoxHft}
\end{eqnarray}
which is defined for every $t>0$, as follows. The indexes $m',n',p',q'$ are given by the forms below,
\begin{eqnarray}
 && m'=n_1+n_2+n_3, \label{mp} \\&&
n'=n_3+n_4, \label{np} \\&&
p'=n_2+n_3+n_4, \label{pp} \\&&
q'=n_1+n_2+n_3+n_4, \label{qp}
\end{eqnarray}
where $n_1$, $n_2$, $n_3$, $n_4$ are natural numbers such that
  \begin{eqnarray}
	n_1 \geq 1, \,\, \textrm{or} \,\, n_3 \geq 1. \label{condn1n2n3n4}
\end{eqnarray}
	The involved parameters are defined as below:
\begin{eqnarray}
&&\hspace{-4em}\left(\alpha'_j,A'_j\right)_1^{p'}\equiv \left(1-r_{j},a''_{j}\right)_1^{n_3}
,\left(1-v_{j},a'''_{j}\right)_1^{n_4}
,\left(d_{j},a'_{j}\right)_1^{n_2}
, \label{FoxHftA}\\
&&\hspace{-4em}\left(\beta'_{j},B'_{j'}\right)_1^{q'}
\equiv \left(b_{j},a_{j}\right)_1^{n_1}, \left(c_{j},a'_{j}\right)_1^{n_2}
,\left(o_{j},a''_{j}\right)_1^{n_3}
, \left(1-w_{j},a'''_{j}\right)_1^{n_4}
. \label{FoxHftB}
\end{eqnarray}
Constraints (\ref{cond_ab}) - (\ref{cond_apppwv}) are requested to be fulfilled. The definition of the Fox $H$-function (\ref{FoxHft}) requires constraint (\ref{cond1}) to hold,
\begin{eqnarray}
A'_j\left(\beta'_{j'}+l\right)\neq B'_{j'}\left(\alpha'_j-l'-1\right),
 \label{cond1prime}
\end{eqnarray}
for every $j=1, \ldots, n'$, $j'=1, \ldots, m'$, and $l,l' \in \mathbb{N}_0$. Let $\chi'$, be the value of the parameter $\chi$, defined by Eq. (\ref{chi}), which characterizes the Fox $H$-function (\ref{FoxHft}). The parameter $\chi'$ is positive, 
\begin{eqnarray}
&&\chi'=\sum_{j=1}^{n_1} a_j+2\sum_{j=1}^{n_3} a^{''}_j>0, \label{chift}
\end{eqnarray}
if condition (\ref{condn1n2n3n4}) holds. Instead, the parameter $\chi'$ vanishes, 
$\chi'=0$, for $n_1=n_3=0$.

Following Refs. \cite{KilSa2010,MathSaxHau2010}, the Fox $H$-function exists and is an analytic function of the complex variable $z$ for every $z\neq 0$ such that $\left|\arg z\right|<\pi\chi/2$. The parameter $\chi'$, given by Eq. (\ref{chift}), is positive for the Fox $H$-function (\ref{FoxHft}) in case condition (\ref{condn1n2n3n4}) holds. Thus, the Fox $H$-function (\ref{FoxHft}) exists and is continuous for every $t\in \mathbb{R}^+$ if condition (\ref{cond1prime}) holds for every $j=1, \ldots, n$, $j'=1, \ldots, m$, $l,l' \in \mathbb{N}_0$, and $n_1\in \mathbb{N}$, or $n_3\in \mathbb{N}$. Under the above-reported conditions, the Mellin transform of the Fox $H$-function (\ref{FoxHft}) is given by the following function \cite{KilSa2010,MathSaxHau2010}: 
\begin{eqnarray}
\Xi_{p',q'}^{m',n'}\left[s\Bigg|
\begin{array}{rr}
\left(\alpha'_{j},A'_{j}\right)_1^{p'}\\
\left(\beta'_{j'},B'_{j}\right)_1^{q'}
\end{array}
\right],  \label{Xisp}
\end{eqnarray}
for every value of the complex variable $s$ which belongs to the strip (\ref{MTHcond1}). This strip coincides with the strip (\ref{stripMft}).

Relations (\ref{cond_ab}) - (\ref{cond_apppwv}), (\ref{mp}) - (\ref{FoxHftB}), and constraint 
(\ref{cond1prime}), required to hold for every $j=1, \ldots, n'$, $j'=1, \ldots, m'$, and $l,l' \in \mathbb{N}_0$, are conditions which are sufficient for the existence and positivity of the Fox $H$-function on $\mathbb{R}^+$. The above-reported conditions are referred to as the e.p. (existence and positivity) conditions, for the sake of shortness.

\begin{theorem}
The Fox $H$-function (\ref{FoxHft}) is positive on $\mathbb{R}^+$,
\begin{eqnarray}
H^{m',n'}_{p',q'}\left[t
\Bigg|
\begin{array}{rr}
\left(\alpha'_{j},A'_{j}\right)_1^{p'}\\
\left(\beta'_{j'},B'_{j}\right)_1^{q'}
\end{array}
\right]> 0, 
\label{FoxHgeq0}
\end{eqnarray}
for every $t>0$, if the e.p. conditions hold.
\label{Hpos}
\end{theorem}

\begin{proof}
The Mellin transform $\hat{f}(s)$, given by Eq. (\ref{MTfGammaProducts}), of the function $f(t)$, defined by Eq. (\ref{ftMellinConvTrunc}) and constraints (\ref{cond_ab}) - (\ref{cond_apppwv}), is equal to the Mellin transform (\ref{Xisp}) of the Fox $H$-function (\ref{FoxHft}),
\begin{eqnarray}
\hat{f}(s)=\Xi_{p',q'}^{m',n'}\left[s\Bigg|
\begin{array}{rr}
\left(\alpha'_{j},A'_{j}\right)_1^{p'}\\
\left(\beta'_{j'},B'_{j}\right)_1^{q'}
\end{array}
\right], \label{fseqHs}
\end{eqnarray}
over the nonempty strip (\ref{stripMft}), if the involved indexes are given by Eqs. (\ref{mp}) - (\ref{qp}) and fulfill constraint (\ref{condn1n2n3n4}), and constraint (\ref{cond1prime}) holds for every $j=1, \ldots, n'$, $j'=1, \ldots, m'$, $l,l' \in \mathbb{N}_0$. Thus, following Ref. \cite{GlaePrudSkorn2006}, the function $f(t)$ is equal to the Fox $H$-function (\ref{FoxHft}) almost everywhere on $\mathbb{R}^+$,
\begin{eqnarray}
 &&\hspace{-2em}
f(t)\doteq
H^{m',n'}_{p',q'}\left[t
\Bigg|
\begin{array}{ll}
\left(\alpha'_j,A'_j\right)_1^{p'} \\  
\left(\beta'_j,B'_j\right)_1^{q'}
\end{array}
\right], 
\label{ftdoteqFoxH}
\end{eqnarray}
for every $t>0$.

The function $t^{x_0}\varphi_{a,b}: \mathbb{R}^+\to \mathbb{R}^+$ is uniformly continuous and bounded on $\mathbb{R}^+$ for $x_0\geq -b/a$. The function $t^{x_0}\phi_{a,b,c}: \mathbb{R}^+\to \mathbb{R}^+_0$ is uniformly continuous and bounded on $\mathbb{R}^+$ for $x_0\geq -b/a$, and $c\geq b+1$. The function $t^{x_0}\psi_{a,b,c}: \mathbb{R}^+\to \mathbb{R}^+$ is uniformly continuous and bounded on $\mathbb{R}^+$ for $c/a>x_0\geq -b/a$. The function $t^{x_0}\eta_{a,b,c}: \mathbb{R}^+\to \mathbb{R}^+_0$ is uniformly continuous and bounded on $\mathbb{R}^+$ for $x_0 < b/a$, and $c\geq b+1$. Thus, following Ref. \cite{ButJan1997}, the Mellin convolution product of any combination, with possible repetitions, of the functions $\varphi_{a_1,b_1}(t),\ldots,$ $\varphi_{a_{n},b_{n}}(t)$, $\phi_{a'_1,c_1,d_1}(t),\ldots$, $\phi_{a'_{n},c_{n},d_{n}}(t)$, $\psi_{a''_1,o_1,r_1}(t), \ldots$, $\psi_{a''_{n},o_{n},r_{n}}(t)$, $\ldots$, $\eta_{a'''_1,v_1,w_1}(t)$, $\ldots$, $\eta_{a'''_{n},v_{n},w_{n}}(t)$ is continuous on $\mathbb{R}^+$ and admits Mellin transform in the nonempty strip (\ref{stripMft}).

The function $f(t)$, given by Eq. (\ref{ftMellinConvTrunc}), and the Fox $H$-function (\ref{FoxHft}) are equal to each other almost everywhere, Eq. (\ref{ftdoteqFoxH}), and are continuous on $\mathbb{R}^+$. Consequently, the above-reported functions are equal to each other on every closed interval which is included in $\mathbb{R}^+$, 
\begin{eqnarray}
f(t)=H^{m',n'}_{p',q'}\left[t
\Big|
\begin{array}{rr}
\left(\alpha'_j,A'_j\right)_1^{p'} \\  
\left(\beta'_j,B'_j\right)_1^{q'}
\end{array}
\right], 
\label{fteqFoxH}
\end{eqnarray}
for every $t,t',t''$, such that $t''\geq t\geq t'>0$. 
\end{proof}

In summary, the Fox $H$-function (\ref{FoxHft}) exists and is continuous on $\mathbb{R}^+$ if condition (\ref{cond1prime}) holds for every $j=1, \ldots, n'$, $j'=1, \ldots, m'$, and $l,l' \in \mathbb{N}_0$, and constraint (\ref{condn1n2n3n4}) is fulfilled \cite{KilSa2010,MathSaxHau2010}. The function $f(t)$ is properly defined and is positive on $\mathbb{R}^+$ if conditions (\ref{cond_ab}) - (\ref{cond_apppwv}) hold and constraint (\ref{condn1n2n3n4}) is fulfilled. The Fox $H$-function (\ref{FoxHft}) is equal on $\mathbb{R}^+$ to the function $f(t)$ if the above-reported conditions hold. Thus, the Fox $H$-function (\ref{FoxHft}) exists and is positive on $\mathbb{R}^+$ if the e.p. conditions hold.

\section{Further forms of positive Fox $H$-functions}\label{5}

In the previous Section, we have obtained conditions involving indexes and parameters, labeled as the e.p. conditions, which make the corresponding Fox $H$-functions positive on $\mathbb{R}^+$. In the present Section, we show how further positive Fox $H$-functions can be determined from positive Fox $H$-functions by using elementary properties or integral transforms. In order to simplify the notation, in the present Section, we use the indexes $m',n',p',q'$, and the parameters $\alpha'_1, \ldots,\alpha'_p$, 
$A'_1, \ldots,A'_p$, $\beta'_1, \ldots,\beta'_q$, $B'_1, \ldots$, $B'_q$, to identify a general Fox $H$-function which is positive on $\mathbb{R}^+$, as the Fox $H$-function (\ref{FoxHft}) is. We consider uniquely real values of the independent variables and the involved parameters since the present analysis concerns positive Fox $H$-functions. However, the properties and the integral transforms reported below can be generalized to complex values of the involved variable and parameters \cite{KilSa2010,MathSaxHau2010}.

The following elementary properties of the Fox $H$-function hold for every allowed values of the involved parameters \cite{KilSa2010,MathSaxHau2010}:
\begin{eqnarray}
H^{m,n}_{p,q}\left[z
\Bigg|
\begin{array}{rr}
\left(\alpha_j,A_j\right)_1^{p} \\ 
\left(\beta_j,B_j\right)_1^{q}
\end{array}
\right]=
H^{n,m}_{q,p}\left[\frac{1}{z}
\Bigg|
\begin{array}{rr}
 \left(1-\beta_j,B_j\right)_1^{q}\\ \left(1-\alpha_j,A_j\right)_1^{p}
\end{array}
\right], 
\label{FoxH12geq0}
\end{eqnarray}
\begin{eqnarray}
H^{m,n}_{p,q}\left[z
\Bigg|
\begin{array}{rr}
\left(\alpha_j,A_j\right)_1^{p} \\
\left(\beta_j,B_j\right)_1^{q}
\end{array}
\right]=\omega
H^{m,n}_{p,q}\left[z^{\omega}
\Bigg|
\begin{array}{rr}
\left(\alpha_j,\omega A_j\right)_1^{p} \\ 
\left(\beta_j,\omega B_j\right)_1^{q}
\end{array}
\right], 
\label{FoxH13geq0}
\end{eqnarray}
\begin{eqnarray}
z^{\pm\omega} H^{m,n}_{p,q}\left[z
\Bigg|
\begin{array}{rr}
\left(\alpha_j,A_j\right)_1^{p} \\
\left(\beta_j,B_j\right)_1^{q}
\end{array}
\right]=
H^{m,n}_{p,q}\left[z
\Bigg|
\begin{array}{rr}
\left(\alpha_{j}\pm\omega A_{j}, A_{j}\right)_1^{p} \\ 
\left(\beta_{j}\pm\omega B_{j}, B_{j}\right)_1^{q}
\end{array}
\right],
\label{FoxH14geq0}
\end{eqnarray}
for $\omega >0$.

\begin{corollary}
\label{c41}
The following Fox $H$-function is positive on $\mathbb{R}^+$,
\begin{eqnarray}
H^{m'',n''}_{p'',q''}\left[t
\Bigg|
\begin{array}{rr}
\left(\alpha''_{j},A''_{j}\right)_1^{p''} \\  
\left(\beta''_{j},B''_{j}\right)_1^{q''}
\end{array}
\right]> 0, 
\label{FoxH1geq0}
\end{eqnarray}
for every $t>0$, if any of the following sets of conditions hold: $m''=n'$, $n''=m'$, $p''=q'$, $q''=p'$, $\left(\alpha''_{j},A''_{j}\right)_1^{q'}=\left(1-\beta'_{j},B'_{j}\right)_1^{q'}$, $\left(\beta''_{j},B''_{j}\right)_1^{p'}=\left(1-\alpha'_{j},A'_{j}\right)_1^{p'}$; or  
$m''=m'$, $n''=n'$, $p''=p'$, $q''=q'$, $\left(\alpha''_{j},A''_{j}\right)_1^{p'}=\left(\alpha'_{j}, \omega A'_{j}\right)_1^{p'}$, $\left(\beta''_{j},B''_{j}\right)_1^{q'}=\left(\beta'_{j},\omega B'_{j}\right)_1^{q'}$; or  
$m''=m'$, $n''=n'$, $p''=p'$, $q''=q'$, $\left(\alpha''_{j},A''_{j}\right)_1^{p'}=\left(\alpha'_{j}\pm \omega A'_{j},  A'_{j}\right)_1^{p'}$, $\left(\beta''_{j},B''_{j}\right)_1^{q'}=\left(\beta'_{j}\pm \omega B'_{j}, B'_{j}\right)_1^{q'}$; for every $\omega>0$.
\end{corollary}

\begin{proof}
 Relation (\ref{FoxH1geq0}) is proved to hold for every $t>0$,
under the first set of conditions by identity (\ref{FoxH12geq0});
or under the second set of conditions by identity (\ref{FoxH13geq0});
or under the third set of conditions by identity (\ref{FoxH14geq0}).
\end{proof}

\subsection{Positive Fox $H$-functions via integral transforms 
}\label{41}

 Various integral transforms involving the Fox $H$-function are provided in Refs. \cite{MathSax1978,SriGuGo1982,PrBrMa1990,KilSa2010,MathSaxHau2010} and references therein. Further forms of Fox $H$-functions which are positive on $\mathbb{R}^+$ can be derived from integral transforms involving positive Fox $H$-functions. For the sake of clarity and shortness, we show below few examples of positive Fox $H$-functions which are obtained via the Laplace transform and the Euler transform of positive Fox $H$-functions, and the Mellin transforms of the product of two positive Fox $H$-functions. Again, we consider uniquely real values of the involved variables and parameters as we are interested in positive values of the Fox $H$-functions.

 The following Laplace transform \cite{PrBrMa1990,KilSa2010,MathSaxHau2010}:
\begin{eqnarray}
&&\int_0^{+\infty}
\exp\left(-s \tau\right)\tau^{\omega-1}
H^{m,n}_{p,q}\left[\tau^{\lambda}
\Bigg|
\begin{array}{ll}
\left(\alpha_j,A_j\right)_1^{p} \\  
\left(\beta_j,B_j\right)_1^{q}
\end{array}
\right]d\tau\nonumber \\
&&=
s^{-\omega}  H^{m,n+1}_{p+1,q}\left[s^{-\lambda}
\Bigg|
\begin{array}{ll}
\left(1-\omega,\lambda\right),\left(\alpha_{j},A_{j}\right)_1^{p}\\  
\left(\beta_{j},B_{j}\right)_1^{q}
\end{array}
\right], 
\label{LFoxH1geq0}
\end{eqnarray}
holds for $\left(m n\right),\chi,\lambda,s>0$, and 
\begin{eqnarray}
\omega+\lambda \min_{l=1,\ldots,m}\left\{\frac{\beta_{l}}{B_{l}} \right\}>0. 
\label{cond1LFoxH1geq0}
\end{eqnarray}

\begin{corollary}
\label{c411}
The following Fox $H$-function is positive on $\mathbb{R}^+$,
\begin{eqnarray}
H^{m',n'+k}_{p'+k,q'}\left[t
\Bigg|
\begin{array}{rr}
\left(1-\omega_{k-j+1},\lambda_{k-j+1}\right)_1^{k},\left(\alpha'_{j},A'_{j}\right)_1^{p'} \\
\left(\beta'_{j},B'_{j}\right)_1^{q'}
\end{array}
\right]> 0, 
\label{FoxHkgeq0L}
\end{eqnarray}
for every $t>0$, and $k \in \mathbb{N}$, if $\left(m' n'\right),\chi',\lambda_j>0$, and
\begin{eqnarray}
\omega_j+\lambda_j \min_{l=1,\ldots,m'}\left\{\frac{\beta'_{l}}{B'_{l}} \right\}>0, 
\label{condlLFoxHkgeq0}
\end{eqnarray}
for every $j=1,\ldots,k$. 
\end{corollary}

\begin{proof}
Relation (\ref{FoxHkgeq0L}) is proved to hold for every $t>0$, and $k=1$, by the Laplace transform of a positive Fox $H$-function, which is evaluated via Eq. (\ref{LFoxH1geq0}), under the required conditions. Let $\chi_k$ be the value of the parameter $\chi$ which characterizes the Fox $H$-function involved in relation (\ref{FoxHkgeq0L}), for every $k\in \mathbb{N}$. The following property: $$\chi_k=\chi'+\sum_{j=1}^k \lambda_j> \chi'>0,$$ 
provides $\chi_k>0$, for every $k\in \mathbb{N}$, as it is required for the iterative application of the Laplace transform (\ref{cond1LFoxH1geq0}). Thus, relations (\ref{FoxHkgeq0L}) and (\ref{condlLFoxHkgeq0}) are obtained by using iteratively, $k$ times, the Laplace transform (\ref{LFoxH1geq0}) and constraint (\ref{cond1LFoxH1geq0}), for every $k\in \mathbb{N}$.
\end{proof}

The following Euler transform \cite{PrBrMa1990,KilSa2010,MathSaxHau2010}:
\begin{eqnarray}
&&\hspace{-2em}\int_0^{t}
\tau^{\omega_1-1}\left(t-\tau\right)^{\lambda_1-1}
H^{m,n}_{p,q}\left[\zeta \tau^{\omega_2}
\left(t-\tau\right)^{\lambda_2}
\Bigg|
\begin{array}{rr}
\left(\alpha_j,A_j\right)_1^{p} \\  
\left(\beta_j,B_j\right)_1^{q}
\end{array}
\right] d\tau
=t^{\omega_1+\lambda_1-1}  
\nonumber \\&&\hspace{-2em} \times \,H^{m,n+2}_{p+2,q+1}\left[\zeta t^{\omega_2+\lambda_2}
\Bigg|
\begin{array}{rr}
\left(1-\omega_1,\omega_2\right),\left(1-\lambda_1,\lambda_2\right),\left(\alpha_{j},A_{j}\right)_1^{p} \\ 
\left(\beta_{j},B_{j}\right)_1^{q},\left(1-\omega_1-\lambda_1,\omega_2+\lambda_2\right)
\end{array}
\right], 
\label{EulerFoxHgeq0}
\end{eqnarray}
holds for every $\lambda_1\in \mathbb{R}$, $\omega_2,\lambda_2 \geq 0$, $\left(m n\right),\chi,\zeta,t>0$, and 
\begin{eqnarray}
&&\omega_1+\omega_2 \min_{l=1,\ldots,m}
\left\{\frac{\beta_{l}}{B_{l}}\right\}>0, 
\label{cond1EFoxHgeq0}\\
&&\omega_1+\lambda_2 \min_{l=1,\ldots,m}
\left\{\frac{\beta_{l}}{B_{l}}\right\}>0. 
\label{cond2EFoxHgeq0}
\end{eqnarray}

\begin{corollary}
\label{c412}
The following Fox $H$-function is positive on $\mathbb{R}^+$,
\begin{eqnarray}
&&H^{m',n'+2k}_{p'+2k,q'+k}\left[t
\Bigg|
\begin{array}{rr}
\left(\alpha''_j,A''_j\right)_1^{p'+2k} \\
\left(\beta'_j,B'_j\right)_1^{q'}, \left(1-\omega_{2j-1}-\lambda_{2j-1},\omega_{2j}+\lambda_{2j}\right)_1^k
\end{array}
\right]> 0, \nonumber \\&&
\label{FoxHgeq0E}
\end{eqnarray}
for every $t>0$, and $k\in \mathbb{N}$, if $\left(m' n'\right),\chi'>0$,
\begin{eqnarray}
&&\hspace{-2em}\left(\alpha''_j,A''_j\right)_1^{2k}= \left(
\left(1-\omega_{2\left(k-j\right)+1},\omega_{2\left(k-j+1\right)}\right),\left(1-\lambda_{2\left(k-j\right)+1},\lambda_{2\left(k-j+1\right)}\right)\right)_1^k, \nonumber
\\
&& \label{aa12euler}\\
&&\hspace{-2em}
\left(\alpha''_j,A''_j\right)_{2k+1}^{p'+2k}= \left(\alpha'_j,A'_j\right)_{1}^{p'}, \nonumber
\end{eqnarray}
$\omega_{2j},\lambda_{2j}>0$, and
\begin{eqnarray}
&&\omega_{2j-1}+\omega_{2j} \min_{
l=1,\ldots,m'}\left\{\frac{\beta'_{l}}{B'_{l}}\right\}>0, 
\label{cond1jEFoxHgeq0}\\
&&\omega_{2j-1}+\lambda_{2j} \min_{l=1,\ldots,m'}\left\{\frac{\beta'_{l}}{B'_{l}}\right\}>0,
\label{cond2jEFoxHgeq0}
\end{eqnarray}
for every $j=1,\ldots,k$. 
\end{corollary}

\begin{proof}
Relation (\ref{FoxHgeq0E}) is proved to hold for every $t>0$, and $k=1$, by the Euler transform of a positive Fox $H$-function, which is evaluated via Eq. (\ref{EulerFoxHgeq0}) under the required conditions. Let $\chi_k$ be the value of the parameter $\chi$ which characterizes the Fox $H$-function involved in relation (\ref{FoxHgeq0E}), for every $k\in \mathbb{N}$. The following property holds: $\chi_k= \chi'>0$, 
for every $k\in \mathbb{N}$, as it is required for the iterative application of the Euler transform (\ref{EulerFoxHgeq0}). Thus, relations (\ref{FoxHgeq0E}) - (\ref{cond2jEFoxHgeq0}) are obtained by using iteratively, $k$ times, the Euler transform (\ref{EulerFoxHgeq0}) and constraints (\ref{cond1EFoxHgeq0}) and (\ref{cond2EFoxHgeq0}), for every $k\in \mathbb{N}$.
\end{proof}

The following Mellin transform of the product of two Fox $H$-functions \cite{KilSa2010,MathSaxHau2010}:
\begin{eqnarray}
&&\hspace{-2em}\int_0^{+\infty}
  \tau^{\omega-1}
H^{m,n}_{p,q}\left[\xi \tau^{\lambda}
\Bigg|
\begin{array}{rr}
\left(\alpha_j,A_j\right)_1^{p} \\  
\left(\beta_j,B_j\right)_1^{q}
\end{array}
\right]
H^{m'',n''}_{p'',q''}\left[\zeta \tau
\Bigg|\begin{array}{rr}
\left(\alpha''_{j},A''_{j}\right)_1^{p''} \\ 
\left(\beta''_{j},B''_{j}\right)_1^{q''}
\end{array}
\right]d\tau 
\nonumber \\&&\hspace{-2em} =   \zeta^{-\omega} H^{m+n'',
n+m''}_{p+q'',q+p''}\left[\xi \zeta^{-\lambda} 
\Big|
\begin{array}{rr}
\left(\alpha'''_{j},A'''_{j}\right)_1^{p+q''}  \\ 
\left(\beta'''_{j},B'''_{j}\right)_1^{q+p''}
\end{array}
\right],
\label{ProdFoxH}
\end{eqnarray}
holds if $\chi,\chi'',\lambda, \xi, \zeta>0$, and
\begin{eqnarray}
&&- \min_{l=1,\ldots,m}\left\{\frac{\beta_{l}}{B_{l}}, \right\}< \min_{l=1,\ldots,n}\left\{\frac{1-\alpha_{l}}{A_{l}} \right\}, \label{cond1FoxHH}\\
&&-\min_{l=1,\ldots,m''}\left\{\frac{\beta''_{l}}{B''_{l}} \right\}<\min_{l=1,\ldots,n''}\left\{\frac{1-\alpha''_{l}}{A''_{l}} \right\}. 
\label{cond2FoxHH}
\end{eqnarray}
The values of the parameter $\omega$ are determined by the following constraints:
\begin{eqnarray}
&&-\lambda \min_{l=1,\ldots,m}\left\{\frac{\beta_{l}}{B_{l}}, \right\}-\min_{l=1,\ldots,m''}\left\{\frac{\beta''_{l}}{B''_{l}}, \right\}<\omega
\nonumber \\&&<
\lambda \min_{l=1,\ldots,n}\left\{\frac{1-\alpha_{l}}{A_{l}} \right\}+\min_{l=1,\ldots,n''}\left\{\frac{1-\alpha''_{l}}{A''_{l}} \right\}.
\label{cond3FoxHH} 
\end{eqnarray}
The involved parameters are
\begin{eqnarray}
&&\hspace{-4em}\left(\alpha'''_{j},A'''_{j}\right)_1^{p+q''}\equiv \left(\alpha_{j},A_{j}\right)_1^{n},
\left(1-\beta''_{j}-\omega 
B''_{j},\lambda B''_{j}\right)_1^{q''}, 
\left(\alpha_{j},A_{j}\right)_{n+1}^{p} , \label{a3A3}\\&& \hspace{-4em}\left(\beta'''_{j},B'''_{j}\right)_1^{q+p''}\equiv \left(\beta_{j},B_{j}\right)_1^{m},\left(1-\alpha''_{j}-\omega 
A''_{j},\lambda A''_{j}\right)_1^{p''}, 
\left(\beta_{j},B_{j}\right)_{m+1}^{q}, \label{b3B3}\\
&&\hspace{-4em}\chi''=\sum_{j=1}^{n''} A''_j-\sum_{j=n''+1}^{p''} A''_j
+\sum_{j=1}^{m''} B''_j-\sum_{j=m''+1}^{q''} B''_j. \label{chi''} 
 \end{eqnarray}
Let the indexes $m'',n'',p'',q'',$ and the parameters $\alpha''_1, \ldots,\alpha''_{p''}$, $A''_1, \ldots,A''_{p''}$, $\beta''_1, \ldots,\beta''_{q''}$, $B''_1, \ldots$, $B''_{q''}$, characterize a general Fox $H$-function which is positive on $\mathbb{R}^+$, and let the parameter $\chi''$ be positive, $\chi''>0$. Then, the corollary reported below holds.
\begin{corollary}
\label{c413}
The following Fox $H$-function is positive on $\mathbb{R}^+$,
\begin{eqnarray}
H^{m'+n'',n'+m''}_{p'+q'',q'+p''}\left[t 
\Bigg|
\begin{array}{ll}
\left(\alpha'''_{j},A'''_{j}\right)_1^{p'+q''} \\ 
\left(\beta'''_{j},B'''_{j}\right)_1^{q'+p''}
\end{array}
\right]> 0,
\label{ProdFoxHgeq0}
\end{eqnarray} 
for every $t>0$, if the involved parameters are given by Eqs. (\ref{a3A3}) and (\ref{b3B3}) with indexes $m=m'$, $n=n'$, $p=p'$, $q=q'$, and parameters $\left(\alpha_{j},A_{j}\right)_1^{p'}= \left(\alpha'_{j},A'_{j}\right)_1^{p'}$, $\left(\beta_{j},B_{j}\right)_1^{q'}= \left(\beta'_{j},B'_{j}\right)_1^{q'}$. Constraints (\ref{cond1FoxHH}) and (\ref{cond2FoxHH}) are required to hold for the involved indexes and parameters.
\end{corollary}

\begin{proof}
  Both the Fox $H$-functions appearing in the left side of Eq. (\ref{ProdFoxH}) are chosen to be positive on $\mathbb{R}^+$. Thus, relation (\ref{ProdFoxHgeq0}) holds for every $t>0$ if the above-required constraints are fulfilled. Particularly, conditions (\ref{cond1FoxHH}) and (\ref{cond2FoxHH}) realize relation (\ref{cond3FoxHH}) which determines the values of the parameter $\omega$.
\end{proof}

The following Mellin transform of the product of two Fox $H$-functions \cite{KilSa2010,MathSaxHau2010}:
\begin{eqnarray}
&&\hspace{-2em}\int_0^{+\infty}
  \tau^{\omega-1}
H^{m,n}_{p,q}\left[\xi \tau^{-\lambda}
\Bigg|
\begin{array}{rr}
\left(\alpha_j,A_j\right)_1^{p} \\  
\left(\beta_j,B_j\right)_1^{q}
\end{array}
\right]
H^{m'',n''}_{p'',q''}\left[\zeta \tau
\Bigg|
\begin{array}{rr}
\left(\alpha''_{j},A''_{j}\right)_1^{p''} \\ 
\left(\beta''_{j},B''_{j}\right)_1^{q''}
\end{array}
\right]d\tau 
\nonumber \\&&\hspace{-2em} =   \zeta^{\omega} H^{m+m'',
n+n''}_{p+p'',q+q''}\left[\xi \zeta^{\lambda} 
\Bigg|
\begin{array}{rr}
\left(\alpha'''_{j'},A'''_{j}\right)_1^{p+p''}  \\ 
\left(\beta'''_{j},B'''_{j}\right)_1^{q+q''}
\end{array}
\right],
\label{ProdFoxH2}
\end{eqnarray}
holds if $\chi,\chi'',\lambda, \xi, \zeta>0$, and conditions (\ref{cond1FoxHH}) and (\ref{cond2FoxHH}) are fulfilled. The values of the parameter $\omega$ are determined by constraint 
(\ref{cond3FoxHH}). The involved parameters are
\begin{eqnarray}
&&\hspace{-4em}\left(\alpha'''_{j},A'''_{j}\right)_1^{p+p''}\equiv 
\left(\alpha_{j},A_{j}\right)_1^{n},
\left(\alpha''_{j}+\omega A''_{j},\lambda A''_{j}\right)_1^{p''}, 
\left(\alpha_{j},A_{j}\right)_{n+1}^{p} , \label{a3A32}\\&& \hspace{-4em}\left(\beta'''_{j},B'''_{j}\right)_1^{q+q''}\equiv \left(\beta_{j},B_{j}\right)_1^{m},
\left(\beta''_{j}+\omega B''_{j},\lambda B''_{j}\right)_1^{q''},
\left(\beta_{j},B_{j}\right)_{m+1}^{q}. \label{b3B32}
\end{eqnarray}

\begin{corollary}
\label{c414}
The following Fox $H$-function is positive on $\mathbb{R}^+$,
\begin{eqnarray}
H^{m'+m'',n'+n''}_{p'+p'',q'+q''}\left[t 
\Bigg|
\begin{array}{rr}
\left(\alpha'''_{j},A'''_{j}\right)_1^{p+p''} \\ 
\left(\beta'''_{j},B'''_{j}\right)_1^{q+q''}
\end{array}
\right]> 0,
\label{Prod2FoxHgeq0}
\end{eqnarray} 
for every $t>0$, if the involved parameters are given by Eqs. (\ref{a3A32}) and (\ref{b3B32}) with indexes $m=m'$, $n=n'$, $p=p'$, $q=q'$, and parameters $\left(\alpha_{j},A_{j}\right)_1^{p'}= \left(\alpha'_{j},A'_{j}\right)_1^{p'}$, $\left(\beta_{j},B_{j}\right)_1^{q'}= \left(\beta'_{j},B'_{j}\right)_1^{q'}$. Again, constraints (\ref{cond1FoxHH}) and (\ref{cond2FoxHH}) are required to hold for the above-reported values of the involved indexes and parameters.
\end{corollary}
The proof of corollary \ref{c414} is similar to the above-reported proof of corollary \ref{c413}. 

The few above-reported examples show how further forms of Fox $H$-functions which are positive on $\mathbb{R}^+$ can be determined from positive Fox $H$-functions by adopting the elementary properties and the integral transforms of one or the product of two positive Fox $H$-functions.

\section{Particular cases}\label{6}
In the present Section, we consider few examples of special functions which are particular cases of the Fox $H$-function. For the sake of clarity, we show how the e.p. conditions presented in Section \ref{3} and the relations reported 
in Section \ref{4} provide further examples of special functions which are positive on $\mathbb{R}^+$. 

The Wright generalized hypergeometric function is a special case of the Fox $H$-function which exists for every nonvanishing complex value of the variable $z$ if $\mu>-1$ \cite{Wrightdef,KilSa2010,MathSaxHau2010,Luc2019},
\begin{eqnarray}
{}_p W_{q-1}\left[-z\Bigg|
\begin{array}{ll}
\left(\alpha_j, A_j\right)_1^{p}\\
\left(\beta_j, B_j\right)_1^{q-1}
\end{array}
\right]=H_{p,q}^{1,p}\left[z
\Bigg|
\begin{array}{rr}
\left(1-\alpha_j,A_j\right)_1^{p}\\  \left(0,1\right),
\left(1-\beta_j,B_j\right)_1^{q-1}
\end{array}
\right]. \label{WrightFoxH}
\end{eqnarray}
\begin{corollary}
\label{WrightP}
The following Wright generalized hypergeometric function is positive on $\mathbb{R}^+$:
 \begin{eqnarray}
{}_{p'} W_{p'}\left[-t\Bigg|
\begin{array}{rr}
\left(v_{j},a'''_{j}\right)_1^{p'}\\
\left(w_{j},a'''_{j}\right)_1^{p'}
\end{array}
\right]>0, \label{WrightFoxHp1}
\end{eqnarray}
for every $t>0$, if the involved parameters $\left(v_{j},a'''_{j}\right)_1^{p'}$ and $\left(w_{j},a'''_{j}\right)_1^{p'}$ fulfill the e.p. conditions with the parameters $a_1=1$, $b_1=0$ and the indexes $n_1=1$, $n_2=n_3=0$, $n_4=p'$. 
\end{corollary}

\begin{proof}
Relation (\ref{WrightFoxHp1}) is proved to hold for every $t>0$, if the Fox $H$-function appearing in Eq. (\ref{WrightFoxH}) is positive on $\mathbb{R}^+$. 
The e.p. conditions provide Fox $H$-functions which are characterized by indexes $m=m'=1$, $n=n'=p=p'$, $q=q'$ and are positive on $\mathbb{R}^+$. The corresponding natural numbers 
$n_1$, $n_2$, $n_3$, $n_4$ are given by the following values:
$n_1=q'-p'$, $n_2=0$, $n_3=p'-q'+1$, $n_4=q'-1$. Condition (\ref{condn1n2n3n4}) provides $q'=p'+1$, or $q'=p'$. The latter case, 
$q'=p'$, is discarded as $\mu=-1$. The former case, $q'=p'+1$,  provides $a_1=1$ and $b_1=0$, and Eq. (\ref{WrightFoxH}) leads to relation (\ref{WrightFoxHp1}). The corresponding Wright generalized hypergeometric function is characterized by $\mu=0>-1$. 
\end{proof}

\begin{corollary}
\label{WrightPk}
The following Wright generalized hypergeometric function is positive on $\mathbb{R}^+$:
 \begin{eqnarray}
&&{}_{p'+k} W_{p'}\left[-t\Bigg|
\begin{array}{ll}\left(\omega_{k-j+1},\lambda_{k-j+1}\right)_1^{k},
\left(v_{j},a'''_{j}\right)_1^{p'}\\
\left(w_{j},a'''_{j}\right)_1^{p'}
\end{array}
\right]>0, \label{WrightFoxHp2}
\end{eqnarray}
for every $t>0$, if the involved parameters $\left(v_{j},a'''_{j}\right)_1^{p'}$ and $\left(w_{j},a'''_{j}\right)_1^{p'}$ fulfill the e.p. conditions with the parameters $a_1=1$, $b_1=0$, and the indexes $n_1=1$, $n_2=n_3=0$, $n_4=p'$. Additionally, it is required that $\omega_j,\lambda_j>0$, for every $j=1,\ldots,k$, and $\sum_{l=1}^k\lambda_l<1$, for every $k \in \mathbb{N}$.
\end{corollary}

\begin{proof}
Relation (\ref{WrightFoxHp2}) is proved to hold for every $t>0$, if the Fox $H$-function appearing in Eq. (\ref{WrightFoxH}) is positive on $\mathbb{R}^+$. The e.p. conditions provide Fox $H$-functions which are characterized by the indexes $m=m'=1$, $n=n'=p=p'+k$, and $q=q'=p'+1$, and are positive on $\mathbb{R}^+$. The additional conditions are obtained from corollary \ref{c411}. The constraints concerning the parameters $\omega_1,\ldots,\omega_k$, $\lambda_1,\ldots,\lambda_k$ are obtained from relation (\ref{condlLFoxHkgeq0}). The latter constraint provides $\mu>-1$.
\end{proof}

\begin{corollary}
\label{c512}
The following Wright generalized hypergeometric function is positive on $\mathbb{R}^+$:
 \begin{eqnarray}
&&{}_{p'+2k} W_{p'+k}\left[-t\Bigg|
\begin{array}{ll}
\left(1-\alpha''_j,A''_j\right)_1^{2k}
,\left(v_{j},a'''_{j}\right)_1^{p'}\\
\left(w_{j},a'''_{j}\right)_1^{p'},\left(\omega_{2j-1}+\lambda_{2j-1},\omega_{2j}+\lambda_{2j}\right)_1^k
\end{array}
\right]>0, \nonumber \\&&\label{WrightFoxHp3}
\end{eqnarray}
for every $t>0$, if the involved parameters $\left(v_{j},a'''_{j}\right)_1^{p'}$ and $\left(w_{j},a'''_{j}\right)_1^{p'}$ fulfill the e.p. conditions with the parameters $a_1=1$, $b_1=0$ and 
the indexes $n_1=1$, $n_2=n_3=0$, $n_4=p'$. The remaining parameters are defined as follows:
$$
\left(1-\alpha''_j,A''_j\right)_1^{2k}=\left(
\left(\omega_{2\left(k-j\right)+1},\omega_{2\left(k-j+1\right)}\right),\left(\lambda_{2\left(k-j\right)+1},\lambda_{2\left(k-j+1\right)}\right)\right)_1^k,
$$
where $\omega_{2j-1},\omega_{2j},\lambda_{2j}>0$, for every $j=1,\ldots,k$, and $k \in \mathbb{N}$.
\end{corollary}

\begin{proof}
Relation (\ref{WrightFoxHp3}) is proved to hold for every $t>0$, if the Fox $H$-function appearing in Eq. (\ref{WrightFoxH}) is positive on $\mathbb{R}^+$. The e.p. conditions
provide Fox $H$-functions which are characterized by indexes $m=m'=1$, $n=n'=p=p'+2k$, and $q=q'=p'+k+1$, and are positive on $\mathbb{R}^+$. The required conditions are obtained from corollary \ref{c412}. The constraints concerning the parameters $\omega_1,\ldots,\omega_k$, $\lambda_1,\ldots,\lambda_k$ are obtained from relations (\ref{cond1jEFoxHgeq0}) and (\ref{cond2jEFoxHgeq0}). The latter constraint provides $\mu>-1$.
\end{proof}

The MacRobert's $E$-function is a special case of the Fox $H$-function \cite{KilSa2010,MathSaxHau2010},
\begin{eqnarray}
E\left[\beta_{1}, \ldots,\beta_{q};\alpha_{1}, \ldots,\alpha_{p-1};z\right]=H^{q,1}_{p,q}\left[z
\Bigg|
\begin{array}{ll}\left(1,1\right),\left(\alpha_j,1\right)_1^{p-1}\\  \left(\beta_j,1\right)_1^{q}
\end{array}
\right]. \label{EFoxH}
\end{eqnarray}

\begin{corollary}
\label{MacRobertP}
The following MacRobert's $E$-function is positive on $\mathbb{R}^+$,
\begin{eqnarray}
E\left[b_1,\ldots,b_{q'-p'},c_1,\ldots, c_{p'-1},o_1; d_1,\ldots,d_{p'-1}; t\right]>0, \label{EMRp}
\end{eqnarray}
for every $t>0$, if the involved parameters fulfill the e.p. conditions with the parameters $a_j=a'_l=a''_1=1$, for every $j=1,\ldots,q'-p'$, $l=1,\ldots,p'-1$, and $r_1=0$, and indexes $n_1=q'-p'$, $n_2=p'-1$, $n_3=1$, $n_4=0$. 
\end{corollary}
\begin{proof}
Relation (\ref{EMRp}) is proved to hold for every $t>0$, if the Fox $H$-function appearing in Eq. (\ref{EFoxH}) is positive on $\mathbb{R}^+$. The e.p. conditions provide Fox $H$-functions which are positive on $\mathbb{R}^+$ and are characterized by indexes $m=m'=q=q'$, $n=n'=1$, $p=p'$ and parameters $a_j=a'_l=a''_1=1$, for every $j=1,\ldots,q'-p'$, $l=1,\ldots,p'-1$, and $r_1=0$.
\end{proof}

The Meijer $G$-function is is a special case of the Fox $H$-function \cite{KilSa2010,MathSaxHau2010},
\begin{eqnarray}
&&\hspace{-2.5em}
G_{p,q}^{m,n}\left[z^{1/\lambda}\Bigg|
\begin{array}{rr}
\alpha_{1}, \ldots,\alpha_{p}\\
\beta_{1}, \ldots,\beta_{q}
\end{array}
\right]=\lambda \, H_{p,q}^{m,n}\left[z\Bigg|
\begin{array}{rr}
\left(\alpha_j,\lambda\right)_1^{p}\\
\left(\beta_j,\lambda\right)_1^{q}
\end{array}
\right], \label{GeqF}
\end{eqnarray}
for every $\lambda>0$.

\begin{corollary}
\label{MeijerP}
The Meijer $G$-function is positive on $\mathbb{R}^+$,
\begin{eqnarray}
G_{p',q'}^{m',n'}\left[t\Bigg|
\begin{array}{rr}
\alpha'_{1}\pm \omega, \ldots, \alpha'_{p'}\pm \omega\\
\beta'_{1}\pm \omega, \ldots, \beta'_{q'}\pm \omega
\end{array}
\right]
>0, \label{Gp}
\end{eqnarray}
for every $\omega,t>0$, if the e.p. conditions hold for the indexes $m=m'$, $n=n'$, $p=p'$, $q=q'$, and the parameters $\alpha'_{1}, \ldots, \alpha_{p'}$, $\beta'_{1}, \ldots, \beta_{q'}$, $A'_1=\ldots=A'_{p'}=B'_1=\ldots=B'_{q'}=\lambda>0$. 
\end{corollary}

\begin{proof}
Relation (\ref{Gp}) is proved to hold for every $t>0$, and $\omega=0$, if the Fox $H$-function appearing in Eq. (\ref{GeqF}) is positive on $\mathbb{R}^+$. The e.p. conditions
provide positive Fox $H$-functions with the required values of the parameters and $\omega=0$. Relation (\ref{Gp}) holds for every $\omega>0$, due to corollary \ref{c41}.
\end{proof}

Further forms of special functions, which are positive on $\mathbb{R}^+$, can be determined from the above-reported particular cases and positive Fox $H$-functions. The above-reported procedures are derived from elementary properties of the Fox $H$-function and integral transforms.

\section{Summary and conclusions 
}\label{7}

 In the present research work, we have considered four elementary functions which are defined via stretched exponential and power laws and are non-negative, uniformly continuous and bounded on $\mathbb{R}^+$. We have considered the Mellin convolution products of every finite combination, with possible repetitions, of the four above-mentioned functions. We have selected the combinations such that the corresponding convolution products are positive on $\mathbb{R}^+$. The selected convolutions products are equal to positive Fox $H$-functions if determined conditions, labeled as the e.p. conditions, hold. Thus, the above-reported method allows to determine Fox $H$-functions which are positive on $\mathbb{R}^+$. The above-reported properties are demonstrated by relying on the non-negativity, uniform continuity and boundedness of the involved functions and the existence and uniqueness of the Mellin inversion, which are realized by the e.p conditions.

Elementary properties and integral transforms allow to determine further forms of positive Fox $H$-functions starting from Fox $H$-functions which are positive on $\mathbb{R}^+$. In conclusion, the present approach allows to determine a class of Fox $H$-functions which are positive on $\mathbb{R}^+$.


\end{document}